\documentclass[12pt]{amsart}
\usepackage{amssymb}
\usepackage{}
\usepackage{cases}

\usepackage{amsmath}
\usepackage{txfonts}
\theoremstyle{plain}
\newtheorem{Thm}{Theorem}

\newtheorem{Lem}[Thm]{Lemma}

\begin{document} 
\title[Brezis theorem for Fractional Laplacian]
{Boundeness of solutions to Fractional Laplacian Ginzburg-Landau equation}

\author{Li Ma}

\address{Li Ma: Zhongyuan Institute of mathematics and Department of mathematics \\
Henan Normal university \\
Xinxiang, 453007 \\
China} \email{lma@tsinghua.edu.cn}

\dedicatory{}
\date{May 26th, 2015}

\begin{abstract}

In this paper, we give the boundeness of solutions to Fractional Laplacian Ginzburg-Landau equation, which extends the Brezis theorem into the nonlinear Fractional Laplacian equation. A related linear fractional Schrodinger equation is also studied.


\end{abstract}

\thanks{$^*$ The research is partially supported by the National Natural Science
Foundation of China (No. 11271111) and SRFDP 20090002110019. }
\maketitle

\section{Introduction}\label{sect1}
In this paper, we continue our study of nonlocal nonlinear elliptic problem with the fractional Laplacian \cite{M}. We give the boundeness of solutions to Fractional Laplacian Ginzburg-Landau equation, which extends the Brezis theorem \cite{B} \cite{M1}\cite{M2}\cite{MW} into the nonlinear Fractional Laplacian equation. The proof of our result depends on a Liouville type theorem for $L^p$ non-negative solutions to a nonlinear fractional Laplacian inequality.

We begin with the definition of fractional Laplacian on $R^n$. Let $0<\alpha<2$. Following \cite{CS} we define
$$E=C^{1,1}_{loc}(R^n)\bigcap L_\alpha,$$
 where
$$
L_\alpha=\{u\in L^1_{loc}(R^n); \int_{R^n}\frac{|u(x)|dx}{1+|x|^{n+\alpha}}<\infty \}.
$$
 For $u\in E$, we define the fractional Laplacian operator $(-\Delta)^{\alpha/2}$ by
$$
(-\Delta)^{\alpha/2}u(x)=C_{n,\alpha}\int_{R^n}\frac{u(x)-u(y)}{|x-y|^{n+\alpha}}dy,
$$
where $C_{n,\alpha}$is the uniform constant \cite{La}. The fractional Ginzburg-Landau equation is
\begin{equation}\label{GL}
(-\Delta)^{\alpha/2}u=u(1-|u|^2),  \ \ \  in \ \ R^n.
\end{equation}

We consider the physical meaningful solutions and our main result is below.
\begin{Thm}\label{main} Let $u\in E$ is a solution to (\ref{GL}) such that
$$
\int_{R^n}(1-|u|^2)^2dx<\infty.
$$
Then, we have
$$
|u(x)|\leq 1, \ \ in \ \ R^n
$$
\end{Thm}
This is an extension of Brezis theorem \cite{MW} about the Ginzburg-Landau equation/system. The result is also true for the corresponding vector-valued solution $u:R^n\to R^N$. It is quiet possible to remove the condition $1-u^2\in L^2$. However, we can give an example of linear fractional Schrodinger equation, which shows that the behavior of solutions to linear equation is also very subtle.

Assume $k(x)\geq 0$ is non-negative smooth function on $R^n$. Let $g_\alpha(x,y)=C_{n,-\alpha}\frac{1}{|x-y|^{n-\alpha}}$.
We also consider non-negative solutions to the following linear fractional Laplacian equation
\begin{equation}\label{nonlocal}
(-\Delta)^{\alpha/2}u+k(x)u=0,  \ \ \  in \ \ R^n.
\end{equation}

We have the below
\begin{Thm}\label{ma} Assume $k(x)\geq 0$ is a nontrivial non-negative smooth function on $R^n$. Assume that for each $x\in R^n$,
\[\label{assum}
\int_{R^n} g_\alpha(x,y)k(x)dy<\infty.
\]
Then there
is a non-trivial non-negative solution to (\ref{nonlocal}).
\end{Thm}

 One remark is given now. We actually only need to assume (\ref{assum}) is true at some point $x$.

 The plan of this note is below. The proof of Theorem \ref{ma} is given in section \ref{sect2} and Theorem \ref{main} is proven in section \ref{sect3}.

\section{Linear equation}\label{sect2}

We now prove Theorem \ref{ma}. Assume (\ref{assum}). We let, for each $x\in R^n$,
\[\label{assum2}
V(x)=\int_{R^n} g_\alpha(x,y)k(x)dy<\infty.
\]
Then $V(x)$ is the minimum non-negative solution to the Poisson equation
$$
(-\Delta)^{\alpha/2}u=k(x), \ \ in \  \ R^n
$$
and we have $\inf_{R^n}V(0)=0$. Let $R^n=\bigcup_{j\geq 1} B_j(0)$ be a ball exhaustion of $R^n$. We denote by $B_j=B_j(0)$.
We solve $u_j(x)\geq 0$ such that
$$
(-\Delta)^{\alpha/2}u_j+k(x)u_j=0, \ \ \ in \ \ B_j
$$
with the boundary condition $u_j=1$ on $B_j^c=R^n-B_j$. By the Maximum principle \cite{CF} \cite{CL} we have $0\leq u_j(x)\leq 1$ on $R^n$. This solution can be obtained by the variation method or the monotone method.
By the comparison lemma we know that $(u_j(x))$ is monotone non-increasing sequence and we may let
$$
U(x)=\lim_{j\to\infty} u_j(x).
$$
Note that $0\leq U(x)\leq $ on $R^n$. We now show that $U$ is non-trivial. Let $\tilde{u}_j=1-u_j$. Then
$$
(-\Delta)^{\alpha/2}\tilde{u}_j=k(x)u_j(x)\leq k(x), \ \ in \ \ B_j
$$
and $\tilde{u}_j(x)=0$ on $B_j^c$. By the Maximum principle we have $\tilde{u}_j(x)\leq V(x)$ on $R^n$. Passing to limit we have
$$
1-U(x)\leq V(x), \ \ on \ \ R^n.
$$
Since $\inf V(x)=0$, we know that $U(x)$ is a non-trivial non-negative solution to (\ref{nonlocal}). This completes the proof of Theorem \ref{ma}.

\section{Proof of Theorem \ref{main}}\label{sect3}

Recall that we have Kato's inequality of the form \cite{FLS}
$$
(-\Delta)^{\alpha/2}|f|(x)\leq sgn(f)(-\Delta)^{\alpha/2}f(x), \  \ a.e. \ \ R^n.
$$
By this we have for any $f\in E$, we have
$$
(-\Delta)^{\alpha/2}f_+(x)\leq sgn(f_+)(-\Delta)^{\alpha/2}f(x), \  \ a.e. \ \ R^n,
$$
where $f_+(x)=\sup(f(x),0)$.

Let $u\in E$ be a solution to (\ref{GL}) such that
$$
\int_{R^n}(1-|u|^2)^2dx<\infty.
$$
Let
$$
Q(x)=|u(x)|^2-1.
$$
Note that
$$
(-\Delta)^{\alpha/2}u^2(x)=C_{n,\alpha}\int_{R^n} \frac{u^2(x)-u^2(y)}{|x-y|^{n+\alpha}}dy
$$
$$
\ \ \ \ \ \ \ \ \ \ \ \ \ \ \ \ =2u(x)(-\Delta)^{\alpha/2}u(x) -C_{n,\alpha}\int_{R^n}\frac{|u(x)-u(y)|^2}{|x-y|^{n+\alpha}}dy.
$$
Then
$$
(-\Delta)^{\alpha/2}u^2(x)\leq 2u(x)(-\Delta)^{\alpha/2}u(x).
$$
By the equation (\ref{GL}) we have
$$
(-\Delta)^{\alpha/2}u^2(x)\leq -2u^2(x)Q(x)=-2Q^2(x)-2Q(x).
$$
That implies that
$$
(-\Delta)^{\alpha/2}Q(x)\leq -2Q^2(x)-2Q(x).
$$

Using the Kato inequality above we have
$$
(-\Delta)^{\alpha/2}Q_+(x)\leq -2Q_+^2(x).
$$
Invoking lemma \ref{Lem1} below (with $q=2$) we can conclude that $Q_+(x)=0$ on $R^n$, which implies Theorem \ref{main}.

\begin{Lem}\label{Lem1} Let $1\leq r<\infty$.
Assume that $0\leq f\in L^q(R^n)$ for some $1\leq q<\infty$ such that
$$
(-\Delta)^{\alpha/2}f+f^{r}\leq 0, \ \ \ in \ \ R^n,
$$
in the distributional sense, i.e.,
\[\label{weak}
\int_{R^n} f(-\Delta)^{\alpha/2}v+\int_{R^n}f^{r}v\leq 0,
\]
for any $v\in C_0^\infty(R^n)$ with $v\geq 0$,
then we have $f=0$ on $R^n$.
\end{Lem}
\begin{proof}
Let $\xi(x)\in C^{1,1}(B_2(0))$ be the cut-off function such that $\xi(x)=1$ on $B_1(0)$.
For any $R>1$, let $\xi_R(x)=\xi(x/R)$. Then (\ref{weak}) implies that
\[\label{weak1}
\int_{R^n} f(-\Delta)^{\alpha/2}v+\xi_R(x)\int_{R^n}f^{r}v\leq 0,
\]
for any $v\in C_0^\infty(R^n)$ with $v\geq 0$,

Define
$$
\phi(x)=\int_{R^n} g_\alpha(x,y)\xi_R(x)dy.
$$
Then $0\leq \phi(x)\leq C$ for some uniform constant $C>0$, $\phi(x)\leq C|x|^{\alpha-n}$ at infinity and
$$
(-\Delta)^{\alpha/2}\phi(x)=\xi_R(x), \ \ \ on \ \ R^n.
$$
Define, for any $p>1$,
$$
W^{\alpha,p}=\{f\in L^p, (-\Delta)^{\alpha/2}f(x)\in L^p\bigcap L^\infty\}.
$$
Then $C_0^\infty(R^n)$ is dense in $W^{\alpha,p}$. If $q=1$, we choose any $p>1$. If $q>1$, we let $p=\frac{q}{q-1}$. By passing to limit, we can take the test function $v$ in $W^{\alpha,p}$ for the inequality (\ref{weak1}).
In particular, we may let
$v=\phi$ and we have
$$
\int_{R^n}f\xi_R(x) +\xi_R(x)f^{r}\phi(x)\leq 0.
$$
Note that each term in the integration is non-negative. Then we have
$$
f\xi_R(x)=0, \ \ \ a.e. \ \ R^n.
$$
Since $R>1$ is arbitrary, we have
$f(x)=0$ a.e. in $R^n$.
\end{proof}

\end{document}